\documentclass{amsart}
\usepackage{amsthm, amssymb, amsmath, pb-diagram,enumerate}

\input xy
\xyoption{all}

\newcommand{\geom}[1]{#1\otimes_k\bar{k}}
\newcommand{\Derc}[1]{D^b_c(\geom{#1})}
\newcommand{\Derm}[1]{D^b_m(#1)}
\newcommand{\Perc}[1]{\Perv_c(\geom{#1})}
\newcommand{\Perm}[1]{\Perv_m(#1)}
\newcommand{\diac}[1]{D_{\Diamond,c}(\geom{#1})}
\newcommand{\diam}[1]{D_{\Diamond,m}(#1)}
\newcommand{\pdiac}[1]{\Perv_{\Diamond,c}(\geom{#1})}
\newcommand{\pdiam}[1]{\Perv_{\Diamond,m}(#1)}
\newcommand{\Tate}[1]{\langle{#1}\rangle}
\newcommand{\leftexp}[2]{{\vphantom{#2}}^{#1}{#2}}

\DeclareMathOperator{\GL}{GL}

\DeclareMathOperator{\Gr}{Gr}
\DeclareMathOperator{\Hom}{Hom}
\DeclareMathOperator{\Ext}{Ext}
\DeclareMathOperator{\End}{End}
\DeclareMathOperator{\Aut}{Aut}

\DeclareMathOperator{\Spec}{Spec}
\DeclareMathOperator{\Perv}{Perv}
\DeclareMathOperator{\id}{id}

\DeclareMathOperator{\codim}{codim}
\DeclareMathOperator{\Frob}{Frob}
\DeclareMathOperator{\Fu}{Frob-unip}
\DeclareMathOperator{\rot}{rot}

\def\AA{\mathbb{A}}

\def\CC{\mathbb{C}}

\def\FF{\mathbb{F}}
\def\GG{\mathbb{G}}
\def\HH{\mathbb{H}}

\def\PP{\mathbb{P}}
\def\QQ{\mathbb{Q}}

\def\ZZ{\mathbb{Z}}

\def\calO{\mathcal{O}}
\def\calA{\mathcal{A}}
\def\calB{\mathcal{B}}

\def\calF{\mathcal{F}}
\def\calH{\mathcal{H}}
\def\calI{\mathcal{I}}
\def\calP{\mathcal{P}}

\def\calL{\mathcal{L}}
\def\calK{\mathcal{K}}
\def\calT{\mathcal{T}}
\def\calQ{\mathcal{Q}}
\def\IC{\mathcal{IC}}

\def\Flag{\mathcal{F}\ell}
\def\Ql{\overline{\mathbb{Q}}_{\ell}}
\def\Qla{\overline{\mathbb{Q}}_{\ell,\alpha}\Tate{\dim X_\alpha}}

\def\coch{\mathbb{X}_\bullet}

\def\tilT{\widetilde{T}}
\def\tilt{\widetilde{t}}
\def\tilu{\widetilde{u}}
\def\tilv{\widetilde{v}}
\def\tilw{\widetilde{w}}
\def\tilW{\widetilde{W}}
\def\tilalpha{\widetilde{\alpha}}
\def\tilIp{\widetilde{I^+}}
\def\tildel{\widetilde{\Delta}}
\def\tilnab{\widetilde{\nabla}}

\swapnumbers
\theoremstyle{plain}
\newtheorem{theorem}[subsubsection]{Theorem}
\newtheorem{lemma}[subsubsection]{Lemma}
\newtheorem{cor}[subsubsection]{Corollary}
\newtheorem{prop}[subsubsection]{Proposition}
\newtheorem*{warning}{Warning}

\theoremstyle{definition}
\newtheorem{defn}[subsubsection]{Definition}
\newtheorem{ex}[subsubsection]{Example}

\theoremstyle{remark}
\newtheorem*{remark}{Remark}

\numberwithin{equation}{subsection}

\begin{document}

\title{Weights of mixed tilting sheaves and geometric Ringel duality}
\author{Zhiwei Yun}
\address{
Department of Mathematics,\\ 
Princeton University,\\ 
Princeton, NJ 08544,\\
U.S.A.}

\email{zyun@math.princeton.edu}



\date{Mar, 2008; Revised Nov, 2008}

\begin{abstract}
We describe several general methods for calculating weights of mixed tilting sheaves. We introduce a notion called ``non-cancellation property'' which implies a strong uniqueness of mixed tilting sheaves and enables one to calculate their weights effectively. When we have a certain Radon transform, we prove a geometric analogue of Ringel duality which sends tilting objects to projective objects. We apply these methods to (partial) flag varieties and affine (partial) flag varieties and show that the weight polynomials of mixed tilting sheaves on flag and affine flag varieties are essentially given by Kazhdan-Lusztig polynomials. This verifies a mixed geometric analogue of a conjecture by W.Soergel in \cite{Sg1}.
\end{abstract}

\maketitle

\section{Introduction}

\subsection{Mixed tilting sheaves}
The goal of the paper is to calculate the weights of mixed tilting sheaves on certain stratified (ind-)schemes. The main examples will be affine flag varieties and their relatives. Let us begin with some general discussion on mixed tilting sheaves. Let $X=\bigsqcup X_\alpha$ be a stratified scheme over $k=\FF_q$. Suppose it satisfies the technical assumption in Section \ref{ss:ass}. {\em Tilting ($\ell$-adic) sheaves} on $X$ are a special kind of perverse sheaves whose restriction and co-restriction to each stratum is a lisse $\Ql$-sheaf placed on the perverse degree. A {\em mixed tilting sheaf} is a mixed $\ell$-adic perverse sheaf which is a tilting sheaf (see Definition \ref{d:mt}).

In the paper \cite{BBM}, the authors proved some fundamental results for tilting sheaves in the non-mixed setting. Suppose the $H^1$ and $H^2$ of each stratum vanish, then for each stratum $X_\alpha$, there exists a tilting sheaf supported on the closure of $X_\alpha$, whose restriction to $X_\alpha$ is the constant perverse sheaf $\overline{\QQ}_{\ell,X_\alpha}[\dim X_\alpha]$ on $X_\alpha$ (such tilting sheaves are called {\em tilting extensions} of the constant perverse sheaf). In Section 1.4 of {\em loc.cit.}, it was shown that among such tilting extensions, there is an indecomposable one which is unique up to (non-unique) isomorphism.

We consider the mixed version of the above statements (under the same assumptions). We show the existence of indecomposable mixed tilting extensions of $\Qla:=\overline{\QQ}_{\ell,X_\alpha}[\dim X_\alpha](\dim X_\alpha/2)$ (Lemma \ref{l:exist}). To obtain a reasonable uniqueness statement, we introduce a notion called ``(weak) non-cancellation property'' (see Definition \ref{d:noncancel}). Roughly speaking a mixed tilting extension $\calT$ of $\Qla$ satisfies this property if the restriction and co-restriction of $\calT$ to each boundary stratum do not have common weights. We will see in Section \ref{s:ex} that indecomposable mixed tilting sheaves on the affine (partial) flag varieties have this property. In Section \ref{ss:nonc}, we obtain a stronger uniqueness statement than in the non-mixed situation: assuming non-cancellation holds for {\em some} indecomposable mixed tilting extension $\calT$ of $\Qla$, then {\em any} indecomposable tilting extension of $\Qla$ is isomorphic to $\calT$, and the isomorphism is unique up to a scalar.

\subsection{Calculation of weights} We will describe three methods for computing the weights of an indecomposable tilting extension $\calT$ of $\Qla$. We collect the punctual weights on each stratum to form weight polynomials (see Section \ref{ss:wp} for definition).

\begin{enumerate}[(1)]
\item (see Section \ref{ss:cal1}) If $\calT$ is Verdier self-dual, then the coefficients of its weight polynomials satisfy a system of triangular linear equations. The non-cancell-\newline ation property of $\calT$ implies a ``non-cancellation property'' of its weight polynomials, which ensures that the solution is unique.

\item (see Section \ref{ss:wpush}) Suppose $f:X\to Y$ is a proper morphism compatible with the stratifications, then, with some extra assumptions, $f_!\calT$ is either zero or a similar mixed tilting extension on $Y$ (Proposition \ref{p:wpush}). We can calculate the weight polynomials of $f_!\calT$ from those of $\calT$. The author learned this idea from R.Bezrukavnikov.

\item (see Section \ref{ss:cal2}) Suppose we have a Radon transform $R_{X\to Y}$ between $X$ and $Y$ (see Section \ref{ss:radon}), we prove that the underlying non-mixed complex of $R_{X\to Y}(\calT)$ is a projective cover of an IC sheaf in a certain subcategory of perverse sheaves on $Y$. We call this phenomenon ``geometric Ringel duality'' (Proposition \ref{p:tiltproj}). From this, we deduce that $\calT$ has the non-cancellation property (Theorem \ref{th:tcan}). Moreover, we can express the weight polynomials of $\calT$ in terms of the mixed stalks of the IC sheaves on $Y$ (Proposition \ref{p:tp}).
\end{enumerate}

The main applications of these methods are to (partial) flag varieties and affine (partial) flag varieties with Schubert stratifications. These varieties are important in geometric representation theory. The case of affine (partial) flag varieties are more complicated because they are infinite dimensional. We construct Radon transforms for these varieties in Section \ref{s:ex} and show that

\begin{theorem}[for precise statement, see Theorem \ref{th:KL}]\label{th:m}
The weight polynomials of the indecomposable mixed tilting sheaves on the flag variety $f\ell_G$ or affine flag variety $\Flag_G$ are essentially given by Kazhdan-Lusztig polynomials.
\end{theorem}

\subsection{Koszul duality}

From the above theorem, we see for $X=f\ell_G$ or $\Flag_G$ with Schubert stratification, and $\calT_\alpha$ an indecomposable tilting extension of $\Qla$ for some Schubert stratum $X_\alpha$, the weights of $\calT_\alpha$ satisfy the following strong estimate (see Section \ref{ss:conW}):

(W) {\em For each $\beta<\alpha$, $i_\beta^*\calT_\alpha$ is a complex of weight $\geq1$ and $i_\beta^!\calT_\alpha$ is a complex of weight $\leq-1$},

\noindent where $i_\beta:X_\beta\hookrightarrow X$ is the inclusion. Condition (W) implies the non-cancellation property. Using the second method mentioned above, we will show that condition (W) also holds for $\calT$ on partial flag varieties and affine partial flag varieties. 

Observe that condition (W) resembles the condition for defining perverse sheaves. In fact, we can define a new $t$-structure on certain derived category of mixed complexes on $X$ whose heart is characterized by the condition (W). The irreducible objects in this heart are precisely indecomposable mixed tilting sheaves. We want to emphasize the parallelism between IC sheaves and indecomposable mixed tilting sheaves. They are both irreducible objects in the hearts of certain $t$-structures. For IC sheaves, the stalks costalks are often pure in weights (in nice cases such as $\Flag_G$) but they sit in various degrees; on the other hand, for indecomposable mixed tilting sheaves, the stalks and costalks of sit in a single degree but do not have pure weights.

Theorem \ref{th:m} and the above observation give numerical evidence for the {\em Koszul duality conjecture} proposed in \cite{B}, Section 1.2. The conjecture states that there is a self-equivalence on a certain mixed version of $D^b(I^0\backslash\Flag_G)$ exchanging IC sheaves and tilting sheaves (here $I^0$ is the pro-unipotent radical of the Iwahori $I$). As we mentioned above, the condition (W) allows us to define a new $t$-structure on the mixed version of $D^b(I^0\backslash\Flag_G)$ which should be the transport of the perverse $t$-structure under the conjectural self-equivalence. In a joint work of R.Bezrukavnikov and the author \cite{BY}, we give a proof of this conjecture, as well as several other forms of the Koszul duality (the equivariant-monodromic duality and parahoric-Whittaker duality), enriching and generalizing the results of \cite{BGS} in a purely geometric way. Therefore, our results in \cite{BY} can be viewed as a categorification of Theorem \ref{th:m}.

\subsection{Related work}

This work is largely inspired by the paper \cite{BBM}. In fact, the Radon transform and Ringel duality for flag varieties were constructed there. In \cite{Na}, D.Nadler described a topological approach to the Ringel duality using Morse theory. The parallel study of tilting modules in the purely representation-theoretic setting was carried out by W.Soergel in \cite{Sg1} and \cite{Sg2}. Theorem \ref{th:m} is a mixed geometric analogue of Conjecture 7.1 in \cite{Sg1}.

\subsection{Notations and convention}
From Section \ref{s:mt} to Section \ref{s:geom}, all schemes are of finite type over a fixed finite field $k=\FF_q$. Let $\bar{k}$ be an algebraic closure of $k$. For a scheme $X$ as above, let $\geom{X}$ denote its geometric fiber. Let $\ell$ be a prime different from $\textup{char}(k)$.

We will consider the following triangulated categories: 

\noindent$\bullet$ $\Derc{X}$ is the bounded derived category of $\Ql$-complexes with constructible cohomology on $\geom{X}$; the heart of the perverse $t$-structure is $\Perc{X}$.

\noindent$\bullet$ $\Derm{X}$ is the bounded derived category of mixed complexes of $\Ql$-sheaves on $X$(cf. Section 5.1 of \cite{BBD}); the heart of the perverse $t$-structure is $\Perm{X}$.

Let $\omega:\Derm{X}\to\Derc{X}$ be the forgetful (or pull-back) functor. For objects $\calF,\calF'\in\Derm{X}$, the hyper-cohomologies 
\begin{equation*}
\HH^*(\geom{X},\omega(\calF)), \HH^*_c(\geom{X},\omega(\calF'))
\end{equation*}
and extension groups 
\begin{equation*}
\Ext^i_{\geom{X}}(\omega(\calF),\omega(\calF')):=\Hom_{\Derc{X}}(\omega(\calF),\omega(\calF')[i])
\end{equation*}
are equipped with natural Frobenius actions. They are NOT to be confused with $\Ext^i_X(\calF,\calF'):=\Hom_{\Derm{X}}(\calF,\calF'[i])$, which is a plain $\Ql$-vector space. Note that we often omit the symbol $\omega$ if no confusion is likely to arise.

All the operations on complexes of sheaves are understood to be derived functors.

We will fix once and for all a square root of $q$ in $\Ql$, hence the half Tate twist $(1/2)$ makes sense. We write $\Tate{n}$ for $[n](n/2)$. By a weight-$w$-twist of a mixed complex $\calF$, we mean the same underlying complex with Frobenius action rescaled by an $\ell$-adic unit $b\in\Ql^{\times}$ with $|\iota(b)|=q^{w/2}$ for any embedding $\iota:\Ql\to\CC$. For $w\in\ZZ$, these are called integer-weight-twists.

For a vector space $V$ with a Frobenius action, we denote the Frobenius invariants and coinvariants by $V^{\Frob}$ and $V_{\Frob}$. We denote by $V^{\Fu}$ the subspace where Frobenius acts unipotently.


\section{Mixed tilting sheaves}\label{s:mt}

\subsection{Assumptions on spaces}\label{ss:ass} By a {\em stratified scheme}, we mean a scheme $X$ of finite type over $k=\FF_q$ with a stratification by {\em connected smooth affine} schemes $X_\alpha$:
\begin{equation*}
X=\bigsqcup_{\alpha\in I}X_\alpha.
\end{equation*}
The finite index set $I$ is partially ordered such that $\alpha\leq\beta$ if and only if $X_\alpha\subset\overline{X_\beta}$.

For each $\alpha\in I$, let $i_\alpha:X_\alpha\hookrightarrow X$ be the inclusion. Let
\begin{eqnarray*}
\Delta_\alpha=i_{\alpha,!}\Qla;\\
\nabla_\alpha=i_{\alpha,*}\Qla;\\
\IC_\alpha=i_{\alpha,!*}\Qla
\end{eqnarray*}
be the standard, costandard and intersection complexes in $\Perm{X}$ (because $i_\alpha$ is affine). Our normalization makes these complexes to be pure of weight 0 on $X_\alpha$.

Let $D_{\Delta,m}(X)$ (resp. $D_{\nabla,m}(X)$) be the full triangulated subcategory of $\Derm{X}$ generated by integer-weight-twists of $\Delta_\alpha$ (resp. $\nabla_\alpha$) for $\alpha\in I$. Let $D_{\Delta,c}(\geom{X})$  and $D_{\nabla,c}(\geom{X})$ be their images in $\Derc{X}$. We will consider the following condition on $X$:

($\Diamond$) {\em The subcategories $D_{\Delta,m}(X)$ and $D_{\nabla,m}(X)$ coincide.}

When ($\Diamond$) holds, the subcategories $D_{\Delta,c}(\geom{X})$ and $D_{\nabla,c}(\geom{X})$ also coincide. We use the more democratic symbols $\diac{X}$ and $\diam{X}$ to denote these subcategories.

We have the following criterion for the condition ($\Diamond$).

\begin{lemma}
The stratified scheme $X$ satisfies ($\Diamond$) if and only if for each $\alpha,\beta\in I$ and $j\in\ZZ$, the sheaf $\omega(\calH^ji_{\beta}^*i_{\alpha,*}\Ql)$ is a lisse $\Ql$-sheaf on $\geom{X_\beta}$ with unipotent geometric monodromy (i.e., it is a successive extension of constant sheaves).
\end{lemma}
\begin{proof}
Suppose ($\Diamond$) holds, then in particular $\nabla_\alpha\in D_{\Delta,m}(X)$. By definition, this means that for any $\beta$, $i^*_\beta\nabla_\alpha$ is a successive extension of shifts and twists of the constant sheaf, which implies that each $\omega(\calH^ji_{\beta}^*i_{\alpha,*}\Ql)$ is a successive extension of constant sheaves.

Conversely, suppose each $\calL=\omega(\calH^ji_{\beta}^*i_{\alpha,*}\Ql)$ has unipotent geometric monodromy. Then it has a unique finite filtration $0\subset\calL_1\subset\calL_2\subset\cdots\subset\calL$ such that $\calL_m/\calL_{m-1}\subset\calL/\calL_{m-1}$ is the largest subsheaf with trivial geometric monodromy. Since the geometric fundamental group $\pi_1(\geom{X_\beta},*)$ is normal in $\pi_1(X_\beta,*)$, this filtration is invariant under the Frobenius. Therefore this filtration realizes the mixed sheaf $\calH^ji_{\beta}^*i_{\alpha,*}\Ql$ as a successive extension of integer-weight-twists of the constant sheaf on $X_\beta$ (note that by \cite{Del}, $\calH^ji_{\beta}^*i_{\alpha,*}\Ql$ always has integer punctual weights). This means $\nabla_\alpha$ is a successive extension of shifts and integer-weight-twists of $\Delta_\beta$, i.e., $\nabla_\alpha\in D_{\Delta,m}(X)$ for all $\alpha$. Hence $D_{\nabla,m}(X)\subset D_{\Delta,m}(X)$. Applying Verdier duality, we get the opposite inclusion, hence $D_{\nabla,m}(X)=D_{\Delta,m}(X)$.
\end{proof}

\begin{cor}\label{c:orbits}
If the stratification of $X$ is given by the orbits under an algebraic group action, then the condition ($\Diamond$) holds. 
\end{cor}

\begin{remark}
It is easy to see that under the condition ($\Diamond$), for any two locally closed subschemes $i:Y\subset Z$ of $X$ which are unions of strata, the functors $i_!,i_*$ send $\diam{Y}$ to $\diam{Z}$ and the functors $i^!,i^*$ send $\diam{Z}$ to $\diam{Y}$. Moreover, $\diam{X}$ inherits a perverse $t$-structure from that of $\Derm{X}$, with heart $\pdiam{X}:=\Perm{X}\cap \diam{X}$. The irreducible objects in $\pdiam{X}$ are integer-weight-twists of $\IC_{\alpha}$ for $\alpha\in I$. Similar remark applies to $\diac{X}$, and we have the perverse heart $\pdiac{X}$.
\end{remark}

\subsection{Mixed tilting sheaves}\label{ss:basic} Basic properties of tilting sheaves in the non-mixed setting are nicely explained in the first section of \cite{BBM}. We prove here some analogous properties in the mixed setting.

Let $X$ be a stratified scheme satisfying the condition ($\Diamond$). Recall that
\begin{defn}\label{d:mt}
A {\em mixed tilting sheaf} on $X$ (with respect to the given stratification) is an object $\calT\in \pdiam{X}$ such that for all $\alpha\in I$, $i_\alpha^*\calT$ and $i_\alpha^!\calT$ are lisse $\Ql$-sheaves on $X_\alpha$ placed in degree $-\dim X_\alpha$.
\end{defn}

It is easy to see that
\begin{lemma}\label{l:fil}
A mixed perverse sheaf $\calT\in \pdiam{X}$ is a mixed tilting sheaf if and only if it is both a successive extension of integer-weight-twists of standard sheaves and a successive extension of integer-weight-twists of costandard sheaves (in which case we also say that $\calT$ has both a $\Delta$-flag and a $\nabla$-flag).
\end{lemma}

Let $Y\subset X$ be a locally closed subscheme which is a union of strata. We want to extend a mixed tilting sheaf on $Y$ to a mixed tilting sheaf on the closure $\overline{Y}$ of $Y$. In \cite{BBM}, Section 1.1, such an existence result is proved in the non-mixed setting. The argument in \textit{loc.cit.} also works to prove

\begin{lemma}\label{l:exist}
Suppose $H^i(\geom{X_\alpha},\Ql)=0$ for $i=1,2$ and all $\alpha\in I$. Then for any perverse sheaf $\calF\in \pdiam{Y}$, there exists a mixed tilting sheaf $\calT\in \pdiam{\overline{Y}}$ such that $\calT|_Y\cong\calF$. Moreover, if $\omega(\calF)\in \pdiac{Y}$ is indecomposable, we can choose $\calT$ such that $\omega(\calT)\in\pdiac{\overline{Y}}$ is also indecomposable.
\end{lemma}
\begin{proof}
By induction on strata, the lemma reduces to the case $Y=X-Z$ where $Z$ is a single closed stratum. Let $i:Z\hookrightarrow X$ and $j:Y\hookrightarrow X$ be the inclusions. Consider the exact sequence
\begin{equation}\label{eq:ext2}
0\to i_*\calA\to j_!\calF\to j_*\calF\to i_*\calB\to0.
\end{equation}
in $\pdiam{X}$. Here $\calA,\calB\in\pdiam{Z}$. The only modification to the argument in \cite{BBM} is that we have to make sure the Yoneda Ext-group $\Ext_{\pdiam{X}}^2(i_*\calB,i_*\calA)$ is 0. Note that by Remark 3.1.17 in \cite{BBD}, the natural map
\begin{equation*}
\Ext_{\pdiam{X}}^2(i_*\calB,i_*\calA)\to\Ext_{X}^2(i_*\calB,i_*\calA)=\Ext_Z^2(\calB,\calA)
\end{equation*}
is injective. Therefore it suffices to show $\Ext_Z^2(\calB,\calA)=0$. Of course this reduces to the case where $\calA$ and $\calB$ are twists of $\Qla$. We have an exact sequence
\begin{equation*}
0\to \Ext_{\geom{Z}}^1(\calB,\calA)_{\Frob}\to\Ext_{Z}^2(\calB,\calA)\to\Ext_{\geom{Z}}^2(\calB,\calA)^{\Frob}\to 0.
\end{equation*}
The vanishing of the first and third term follows from the fact that $H^i(Z,\Ql)=0$ for $i=1,2$, Therefore the middle term also vanishes.

Now, since the Yoneda extension (\ref{eq:ext2}) is trivial, we can find an object $\calT\in\pdiam{X}$ with exact sequences
\begin{eqnarray*}
0\to i_*\calA\to\calT\to j_*\calF\to0\\
0\to j_!\calF\to\calT\to i_*\calB\to0
\end{eqnarray*}
and an obvious morphism between the two sequences. In particular, $i^!\calT=\calA$, $i^*\calT=\calB$, and the natural map $i^!\calT\to i^*\calT$ is zero.

Now suppose $\omega(\calF)$ is indecomposable. If $\omega(\calT)$ is decomposable, it must contain a direct summand $\calK\in\pdiac{X}$ which is supported on $Z$. But then $i^!\calT\to i^*\calT$ cannot be zero because it contains a direct summand isomorphic to $\id_{\calK}$. Hence $\omega(\calT)$ is also indecomposable.
\end{proof}

\begin{warning}
In the following, when we say a mixed tilting sheaf $\calT\in\pdiam{X}$ is {\em indecomposable}, we always mean that the non-mixed complex $\omega(\calT)\in\pdiac{X}$ is indecomposable.
\end{warning}

\begin{remark}
In the non-mixed setting, we have the following uniqueness statement (cf. Section 1.4 in \cite{BBM}): if $H^1(\geom{X_\beta},\Ql)=0$ for all $\beta$, then the indecomposable tilting extension of the constant perverse sheaf $\Ql[\dim X_\alpha]$ on $X_\alpha$ is unique up to non-unique isomorphisms in $\pdiac{X}$. In the mixed setting, we will see in the next section that under certain conditions, the indecomposable mixed tilting extension of $\Qla$ is unique up to a unique isomorphism in $\pdiam{X}$.
\end{remark}

\subsection{Non-cancellation property}\label{ss:nonc}

\begin{defn}\label{d:noncancel}
Let $\calT$ be a mixed tilting extension of $\Qla$. We say that $\calT$ has the {\em weak non-cancellation property} if for each $\beta<\alpha$, the mixed perverse sheaves $i_\beta^*\calT$ and $i_\beta^!\calT$ do not have isomorphic simple subquotients (equivalently, they have no simple subquotients of the same Frobenius eigenvalue). We say that $\calT$ has the {\em non-cancellation property} if for each $\beta<\alpha$, $i_\beta^*\calT$ and $i_\beta^!\calT$ do not have common punctual weights.
\end{defn}

\begin{prop}\label{p:noncancel}
Suppose $H^1(\geom{X_\beta},\Ql)=0$ for all $\beta$. Let $\calT$ be a mixed tilting extension of $\Qla$. The following conditions are equivalent:
\begin{enumerate}[(1)]
 \item $\End_X(\calT)=\Ql$;
 \item $\End_{\geom{X}}(\calT)^{\Fu}=\Ql$;
 \item $\calT$  satisfies the weak non-cancellation property.
\end{enumerate}
\end{prop}
\begin{proof}
We first prove (1)$\Longleftrightarrow$(2). Clearly (2) implies (1). We show (1) also implies (2). Suppose $\End_X(\calT)=\Ql$ but $\dim_{\Ql}\End_{\geom{X}}(\calT)^{\Fu}>1$, then there exists $\phi\in\End_{\geom{X}}(\calT)^{\Fu}$ such that
\begin{equation}\label{eq:fruni}
\Frob(\phi)=\phi+c\cdot\id_{\calT}
\end{equation}
for some $c\in\Ql^{\times}$. But $\phi|_{X_\alpha}=a\cdot\id$ for some $a\in\Ql$. If we restrict (\ref{eq:fruni}) to $X_\alpha$, we get a contradiction.

Next, we prove (2)$\Longleftrightarrow$(3). By Lemma \ref{l:fil}, we can write $\calT$ as a $\Delta$-flag or a $\nabla$-flag. Because $H^1(\geom{X_\beta},\Ql)=0$ for all $\beta$, we have
\begin{equation*}
\Ext^1_{\geom{X}}(\Delta_\beta,\nabla_\gamma)=0, \forall\beta,\gamma\in I.
\end{equation*}
Therefore $\End_{\geom{X}}(\calT)$ is a successive extension of $\Hom_{\geom{X}}(\tildel_\beta,\tilnab_\gamma)$ for those twists $\tildel_\beta$ of $\Delta_\beta$ (resp. twists $\tilnab_\gamma$ of $\nabla_\gamma$) that appear as the subquotients of the $\Delta$-flag (resp. $\nabla$-flag). In particular, $\End_{\geom{X}}(\calT)^{\Fu}$ is a successive extension of these relevant $\Hom_{\geom{X}}(\tildel_\beta,\tilnab_\gamma)^{\Fu}$. Note that
\begin{equation*}
\Hom_{\geom{X}}(\tildel_\beta,\tilnab_\gamma)^{\Fu}=\left\{\begin{array}{ll} \Ql, & \textup{if } \beta=\gamma,\tildel_\beta|_{X_\beta}=\tilnab_\beta|_{X_\beta} \\ 0, & \textup{otherwise.} \end{array} \right.
\end{equation*}
Therefore condition (2) $\Longleftrightarrow$ $\Hom_{\geom{X_\alpha}}(\Delta_\alpha,\nabla_\alpha)$ is the only contribution to \newline $\End_{\geom{X}}(\calT)^{\Fu}$ $\Longleftrightarrow$ For each $\beta<\alpha$, twists $\tildel_\beta$ and $\tilnab_\beta$ which are the same on $X_\beta$ do not both occur in the $\Delta$-flag and the $\nabla$-flag $\Longleftrightarrow$ condition (3).

\end{proof}

By similar argument, we have

\begin{prop}\label{p:strong}
Suppose $H^1(\geom{X_\beta},\Ql)=0$ for all $\beta$. Let $\calT$ be a mixed tilting extension of $\Qla$. Then $\calT$ satisfies the non-cancellation property if and only if the Frobenius weights on $\End^0_{\geom{X}}(\calT)$ are nonzero, where $\End^0_{\geom{X}}(\calT)=\ker(\End_{\geom{X}}(\calT)\to\End_{\geom{X_\alpha}}(\calT|_{X_\alpha}))$.
\end{prop}

\begin{prop}\label{p:can}
Suppose $H^i(\geom{X_\beta},\Ql)=0$ for $i=1,2$ and all $\beta$. Let $\calT$ be an indecomposable mixed tilting extension of $\Qla$. Assume $\calT$ has the weak non-cancellation property. Then any indecomposable mixed tilting extension of $\Qla$ is isomorphic to $\calT$. In particular, $\calT$ is Verdier self-dual.
\end{prop}
\begin{proof}
Let $\calT'$ be another indecomposable mixed tilting extension of $\Qla$. By the remark following Lemma \ref{l:exist}, we see $\omega(\calT')\cong\omega(\calT)$. Recall the following:

\begin{lemma}[\cite{BBD}, 5.5.1] The natural functor sending perverse sheaves on $X$ to pairs $(\calF,\phi)$ where $\calF$ is a perverse sheaf on $\geom{X}$ and $\phi: \Frob_X^*\calF\stackrel{\sim}{\to}\calF$ is fully faithful.
\end{lemma}

From this we easily deduce that the set of mixed structures on $\omega(\calT)$ is a subset of $H^1(\ZZ\Frob,\Aut_{\geom{X}}(\calT))$. Since we require the mixed perverse sheaf to be $\Qla$ on $X_\alpha$, it suffices to show that $H^1(\ZZ\Frob,\Aut^1)$ is trivial where $\Aut^1=\ker(\Aut_{\geom{X}}(\calT)\to\Aut_{\geom{X_\alpha}}(\calT|_{X_\alpha}))$. Note that by the construction of Lemma \ref{l:exist}, $\Aut^1$ is the $\Ql$-points of a {\em unipotent} algebraic group with Lie algebra $\End^0:=\End^0_{\geom{X}}(\calT)$. By the argument of Proposition \ref{p:noncancel}, $\End^0$ has a filtration by ideals with subquotients $E_j=\Hom_{\geom{X}}(\tildel_\beta,\tilnab_\beta)$ (viewed as abelian Lie algebras). Here $\tildel_\beta$ and $\tilnab_\beta$ are the subquotients of a $\Delta$-flag and a $\nabla$-flag of $\calT$. Similarly, $\Aut^1$ has a filtration by normal subgroups with subquotients $E_j$ (viewed as additive groups). By the weak non-cancellation property, this is $\Ql$ with nontrivial Frobenius action. Hence $H^1(\ZZ\Frob,E_j)=(E_j)_{\Frob}=0$ for all $j$. Therefore $H^1(\ZZ\Frob, \End^0)$ and $H^1(\ZZ\Frob, \Aut^1)$ also vanish.
\end{proof}

\begin{remark}
In the situation of the above proposition, we can speak about {\em the} indecomposable mixed tilting extension $\calT_\alpha$ of $\Qla$, which is unique up to a unique isomorphism which restricts to the identity on $X_\alpha$.
\end{remark}

\subsection{Proper Push-forward of tilting sheaves}\label{ss:push} This section serves solely as a preliminary for Section \ref{ss:wpush}. We work in the non-mixed setting.

A morphism $f:X\to Y$ between stratified schemes
\begin{equation*}
X=\bigsqcup_{\alpha\in I}X_\alpha;\hspace{1cm}Y=\bigsqcup_{\beta\in J}Y_{\beta}
\end{equation*}
is said to be {\em compatible with the stratifications} if there exists a map $\phi:I\to J$ such that
\begin{equation*}
f^{-1}(Y_\beta)=\bigsqcup_{\alpha\in\phi^{-1}(\beta)}X_\alpha
\end{equation*}
and each restriction $f_\alpha:X_\alpha\to Y_{\phi(\alpha)}$ is an \'{e}tale locally trivial fibration (necessarily with affine fibers since $X_\alpha$ is affine).

The author learned about the following result from R.Bezrukavnikov.

\begin{prop}\label{p:push}
Suppose $X$ and $Y$ are stratified schemes and $X$ satisfies condition ($\Diamond$). Let $f:X\to Y$ be a proper morphism compatible with the stratifications. Then for any tilting sheaf $\calT\in\pdiac{X}$, $f_*\calT\in\Derc{Y}$ is also a tilting sheaf on $Y$ with respect to the stratification of $Y$. 
\end{prop}
\begin{proof}
We first prove a lemma.
\begin{lemma}\label{l:push} 
Suppose we are in the same situation as above except that $f$ is not assumed to be proper.
\begin{enumerate}[(1)]
\item If $\calF\in\Perc{X}$ has a $\Delta$-flag, then $f_!\calF\in\leftexp{p}{D}^{\geq0}_c(\geom{Y},\Ql)$.
\item Dually, if $\calF\in\Perc{X}$ has a $\nabla$-flag, then $f_*\calF\in\leftexp{p}{D}^{\leq0}_c(\geom{Y},\Ql)$.
\end{enumerate}
\end{lemma}
\begin{proof}
We only need to prove (1); the proof of (2) is similar. Since $\calF$ is a successive extension of $\Delta_\alpha$, $f_!\calF$ is a successive extension by $f_!\Delta_\alpha=f_{\alpha,!}\Qla$. It suffices to show that each $f_{\alpha,!}\Qla\in\leftexp{p}{D}^{\geq0}_c(\geom{Y},\Ql)$. Since $f_{\alpha}$ has affine fibers, we can apply the argument of \cite{BBD}, Corollaire 4.1.2. 
\end{proof}

Now we prove the proposition. We first show that $f_*\calF$ is perverse. Since $f$ is compatible with the stratifications, $f_*\calT$ is constructible with respect to the stratification of $Y$. Lemma \ref{l:push}(1) implies $f_!\calT\in\leftexp{p}{D}^{\geq0}_c(\geom{Y},\Ql)$, and Lemma \ref{l:push}(2) implies $f_*\calT\in\leftexp{p}{D}^{\leq0}_c(\geom{Y},\Ql)$, hence $f_!\calT=f_*\calT\in\Perc{Y}$.

Next we prove that $f_*\calT$ is tilting. For any $\beta\in J$, let
\begin{equation*}
f_{\phi^{-1}(\beta)}:f^{-1}(Y_\beta)=\bigsqcup_{\alpha\in\phi^{-1}(\beta)}X_\alpha\to Y_\beta
\end{equation*}
be the restriction of $f$. Let $i_\beta$, $i_{\phi^{-1}(\beta)}$ be the inclusions $Y_\beta\hookrightarrow Y$ and $f^{-1}(Y_\beta)\hookrightarrow X$. Since $\calT$ has a $\Delta$-flag, $i^*_{\phi^{-1}(\beta)}\calT$ also has a $\Delta$-flag. Applying Lemma \ref{l:push}(1) to $f_{\phi^{-1}(\beta)}$, and by proper base change, we conclude that
\begin{equation*}
i_\beta^*f_!\calT=f_{\phi^{-1}(\beta),!}i_{\phi^{-1}(\beta)}^*\calT\in\leftexp{p}{D}^{\geq0}_c(\geom{Y_\beta},\Ql).
\end{equation*}
But we already know that $f_!\calT$ is perverse, which means $i_\beta^*f_!\calT\in\leftexp{p}{D}^{\leq0}_c(\geom{Y_\beta},\Ql)$. Hence we have $i_\beta^*f_!\calT\in\Perc{Y_\beta}$. 

A dual argument shows that $i_\beta^!f_*\calT\in\Perc{Y_\beta}$, therefore $f_!\calT=f_*\calT$ is a tilting sheaf.  
\end{proof}

\begin{remark}
If each $f_{\alpha}:X_\alpha\to Y_\beta$ is a trivial fibration and if $X$ and $Y$ satisfy ($\Diamond$), the tilting sheaf $f_*\calT$ is in $\pdiam{Y}$.
\end{remark}

Applying Proposition \ref{p:push} to the case where $Y$ is a point, we get
\begin{cor}
For a stratified proper scheme $X$ and any tilting sheaf $\calT\in D^b_c(\geom{X})$, we have
\begin{equation*}
\HH^i(\geom{X},\calT)=0,\forall i\neq0.
\end{equation*}
\end{cor}


\section{Weights of mixed tilting sheaves}

\subsection{Weight polynomials}\label{ss:wp} Suppose the stratified scheme $X$ satisfies the condition ($\Diamond$). It is easy to see that the Grothendieck group
\begin{eqnarray*}
K(\diam{X})\cong K(\pdiam{X})&\cong&\bigoplus_{\alpha\in I}K(\pdiam{X_\alpha})[\Delta_\alpha]\\
&\to&\bigoplus_{\alpha\in I}\ZZ[t,t^{-1}][\Delta_\alpha].
\end{eqnarray*}
Here the $K(\pdiam{X_\alpha})\to\ZZ[t,t^{-1}]$ sends $\Qla$ to 1 and its weight-$n$-twists to $t^n$. For an object $\calF\in\diam{X}$, we write $[\calF]$ for the image of $\calF$ in $\bigoplus_{\alpha\in I}\ZZ[t,t^{-1}][\Delta_\alpha]$, we have:
\begin{equation*}
[\calF]=\sum_{\alpha\in I}W_\alpha(\calF,t)[\Delta_\alpha].
\end{equation*}
Here $W_\alpha(\calF,t)\in\ZZ[t,t^{-1}]$ is called the {\em weight polynomial of $\calF$ along the stratum $X_\alpha$}.

\subsection{Calculation of weights I--linear equations}\label{ss:cal1} Let $\calT$ be a mixed tilting extension of $\Qla$ which is Verdier self-dual. The definition of tilting sheaves implies that $W_\beta(\calT,t)$ has non-negative coefficients. We have the self-duality equation:
\begin{equation}\label{eq:selfdual}
\sum_{\beta\leq\alpha}W_\beta(\calT,t)[\Delta_\beta]=\sum_{\beta\leq\alpha}W_\beta(\calT,t^{-1})[\nabla_\beta]
\end{equation}
and the initial value condition $W_\alpha(\calT_{\alpha},t)=1$.

If we express $[\nabla_\beta]$ in terms of $\ZZ[t,t^{-1}]$-combinations of $[\Delta_\gamma]$ (for $\gamma\leq\beta$), we can compare the coefficients of $[\Delta_\beta]$ in equation (\ref{eq:selfdual}) and get a system of linear equations
\begin{equation}\label{eq:tri}
F_\beta=0, \beta\leq\alpha.
\end{equation}
This system of equations is triangular in the sense that $F_\beta$ only involves the coefficients of $W_\gamma(\calT,t)$ for $\gamma\geq\beta$.

If $\calT$ has the non-cancellation property, then for any $\beta<\alpha$ and integer $i$, $W_\beta(\calT,t)$ does not have non-zero coefficients for $t^i$ and $t^{-i}$ simultaneously. When this holds, we say that $W_\beta(\calT,t)$ has the {\em non-cancellation property}. In particular, $W_\beta(\calT,t)$ has no constant term for $\beta<\alpha$.

The following proposition guarantees that we can solve the triangular system of equations (\ref{eq:tri}) uniquely.
\begin{prop}\label{p:uni}
The self-duality equation
\begin{equation}\label{eq:sd}
\sum_{\beta\leq\alpha}W_\beta(t)[\Delta_\beta]=\sum_{\beta\leq\alpha}W_\beta(t^{-1})[\nabla_\beta]
\end{equation}
has at most one solution $\{W_\beta(t)\in\ZZ_{\geq0}[t,t^{-1}]\}_{\beta\leq\alpha}$ satisfying the non-cancellation property and the initial value condition $W_\alpha(t)=1$.
\end{prop}
\begin{proof}
Suppose we have two different solutions $\{W_\beta(t)\}$ and $\{W'_\beta(t)\}$ with the required properties. Consider their difference $U_\beta(t)=W_\beta(t)-W'_\beta(t)$, which also satisfies the equation (\ref{eq:sd}). Let $\beta$ be a maximal element for which $U_\beta(t)\neq0$. Since $U_\alpha(t)=0$ by initial conditions, we have $\beta<\alpha$. Comparing the coefficients of $[\Delta_\beta]$ on both sides of the equation (\ref{eq:sd}), we conclude that
\begin{equation*}
U_\beta(t)=U_\beta(t^{-1})
\end{equation*}
Now both sides must have a term $ct^n$ for some $c\in\ZZ-\{0\}$ and $n\in\ZZ$. If $c>0$, both $t^n$ and $t^{-n}$ appear in $W_\beta(t)$; if $c<0$, both $t^n$ and $t^{-n}$ appear in $W'_\beta(t)$: in any case, it contradicts the non-cancellation property of $W_\beta(t)$ or $W'_\beta(t)$.
\end{proof}

\subsection{A condition on weights}\label{ss:conW} Let $\calT$ be a mixed tilting extension of $\Qla$. We consider the following condition on the weights of $\calT$:

(W) {\em For each $\beta<\alpha$, $i_\beta^*\calT$ is of weights $\geq1$ and $i_\beta^!\calT$ is of weights $\leq-1$}.

Note that here ``weights'' means weights of complexes, e.g., $\overline{\QQ}_{\ell,\beta}\Tate{\dim X_\beta}$ has weight 0.

\begin{remark}
Clearly, if $\calT$ satisfies the condition (W), then it has the non-cancellation property, hence all results of Section \ref{ss:nonc} apply. In particular, such $\calT$ is unique up to an isomorphism (which is unique up to a scalar), and is Verdier self-dual.
\end{remark}

\begin{lemma}\label{l:wprime}
Suppose $\calT$ is Verdier self-dual. Then the condition \textup{(W)} is equivalent to the condition

\textup{(W')} For each $\beta<\alpha$, $W_\beta(\calT,t)\in t\ZZ[t]$ for each $\beta<\alpha$.
\end{lemma}
\begin{proof}
Since $\calT$ is Verdier self-dual, (W) is equivalent to the condition that $i_\beta^*\calT$ is of weight $\geq1$ for each $\beta<\alpha$, which is obviously equivalent to (W').
\end{proof}

\subsection{Calculation of weights II--push-forward}\label{ss:wpush}
We consider the mixed version of the set-up of Section \ref{ss:push}. Recall $f:X\to Y$ is a proper morphism between stratified schemes which is compatible with the stratifications. We assume $X$ and $Y$ both satisfy the condition ($\Diamond$). We further suppose that each $f_\alpha:X_\alpha\to Y_{\phi(\alpha)}$ is a trivial fibration with {\em affine spaces as fibers}. By the remark following Proposition \ref{p:push}, $f_*$ sends $\diam{X}$ to $\diam{Y}$.

\begin{prop}\label{p:wpush}
Let $\calT_\alpha$ be {\em the} mixed tilting extension of $\Qla$ satisfying the condition \textup{(W)}. Then
\begin{enumerate}[(1)]
\item If $f_\alpha$ is an isomorphism, then $f_!\calT_\alpha$ is {\em the} mixed tilting extension of the constant perverse sheaf $\Ql\Tate{\dim Y_{\phi(\alpha)}}$ on $Y_{\phi(\alpha)}$ which also satisfies the condition (W);
\item If $f_\alpha$ is not an isomorphism (i.e., $\dim X_\alpha>\dim Y_{\phi(\alpha)}$), then $f_!\calT_\alpha=0$.
\end{enumerate}
\end{prop}
\begin{proof}
The functor $f_!$ induces a homomorphism
\begin{equation*}
f_\#:K(\diam{X})\to K(\diam{Y}).
\end{equation*}
Since each $f_\gamma$ is a trivial fibration with affine spaces as fibers, 
\begin{eqnarray*}
f_!\Delta_\gamma&=&i_{\phi(\gamma),!}f_{\gamma,!}\Ql\Tate{\dim X_\gamma}\\
&=&i_{\phi(\gamma),!}\Ql\Tate{-2(\dim X_\gamma-\dim Y_{\phi(\gamma)})+\dim X_\gamma}\\
&=&\Delta_{\phi(\gamma)}\Tate{-\dim X_\gamma+\dim Y_{\phi(\gamma)}}.
\end{eqnarray*}
Therefore
\begin{equation*}
f_\#[\Delta_\gamma]=(-t)^{\dim X_\gamma-\dim Y_{\phi(\gamma)}}[\Delta_{\phi(\gamma)}].
\end{equation*}
Applying $f_\#$ to $[\calT_\alpha]$, we get
\begin{equation*}
[f_!\calT_\alpha]=f_\#[\calT_\alpha]=\sum_{\gamma\leq\alpha}W_\gamma(\calT_\alpha,t)\cdot(-t)^{\dim X_\gamma-\dim Y_{\phi(\gamma)}}[\Delta_{\phi(\gamma)}].
\end{equation*}
Therefore we find
\begin{equation}\label{eq:ww}
W_\beta(f_!\calT_\alpha,t)=\sum_{\gamma\in\phi^{-1}(\beta),\gamma\leq\alpha}W_\gamma(\calT_\alpha,t)\cdot(-t)^{\dim X_\gamma-\dim Y_{\beta}}.
\end{equation}

We distinguish two cases: 

(1) If $f_\alpha$ is an isomorphism, then $f_!\calT_\alpha|_{Y_{\phi(\alpha)}}=\Ql\Tate{\dim Y_{\phi(\alpha)}}$. We know from the mixed version of Proposition \ref{p:push} that $f_!\calT_\alpha$ is a mixed tilting extension of the constant perverse sheaf $\Ql\Tate{\dim Y_{\phi(\alpha)}}$ on $Y_{\phi(\alpha)}$. Note that since $W_\gamma(\calT_\alpha,t)\in t\ZZ[t]$ whenever $\gamma<\alpha$, hence $W_{\beta}(f_!\calT_\alpha,t)\in t\ZZ[t]$ whenever $\beta<\phi(\alpha)$ by (\ref{eq:ww}) (note that the exponent $\dim X_\gamma-\dim Y_{\beta}\geq0$). Since $\calT_\alpha$ is Verdier self-dual, $f_!\calT_\alpha=f_*\calT_\alpha$ is also Verdier self-dual. Therefore by Lemma \ref{l:wprime}, $f_!\calT_\alpha$ satisfies the condition (W).

(2) If $\dim X_\alpha>\dim Y_{\phi(\alpha)}$, then $W_{\beta}(f_!\calT_\alpha,t)\in t\ZZ[t]$ for all $\beta\in J$. Suppose $\beta$ is a maximal index for which $W_{\beta}(f_!\calT_\alpha,t)$ is nonzero. Note that $f_!\calT_\alpha=f_*\calT_\alpha$ is Verdier self-dual. By comparing the coefficients of $[\Delta_\beta]$ in the self-duality equation (\ref{eq:selfdual}), we find $W_{\beta}(f_!\calT_\alpha,t)=W_{\beta}(f_!\calT_\alpha,t^{-1})$. This is impossible. Therefore all the weight polynomials of $f_!\calT_\alpha$ are zero, hence $f_!\calT_\alpha=0$.
\end{proof}

Applying Proposition \ref{p:wpush} to the case where $Y$ is a point, we get:
\begin{cor}
Suppose $X$ is a proper scheme stratified by affine spaces and satisfies ($\Diamond$). Let $\calT_\alpha$ be the mixed tilting extension of $\Qla$ (for some stratum $X_\alpha$) satisfying the condition (W). Then
\begin{equation*}
\HH^*(\geom{X},\calT_\alpha)=0
\end{equation*}
unless $\dim X_\alpha=0$.
\end{cor}


\section{Geometric Ringel duality}\label{s:geom}
In this section, we describe a situation where the non-cancellation property for indecomposable mixed tilting extensions is guaranteed. This situation arises when there exists a certain Radon transform, and resembles the Ringel duality in the sense that the Radon transform sends tilting objects to projective objects.

\subsection{The Radon transform}\label{ss:radon}
Let $B$ be an algebraic group containing a split torus $T$. Let $X,Y$ be schemes acted upon by $B$ with finitely many orbits:
\begin{equation*}
X=\bigsqcup_{\alpha\in I}X_\alpha;\hspace{1cm}Y=\bigsqcup_{\beta\in J}Y_{\beta}.
\end{equation*}

By Corollary \ref{c:orbits}, the stratified schemes $X$ and $Y$ satisfy the condition ($\Diamond$).

Let $U$ be a $B$-stable open subscheme of $X\times Y$, viewed as a correspondence between $X$ and $Y$:
\begin{equation*}
\xymatrix{& U\ar[dl]_{\overleftarrow{u}}\ar[dr]^{\overrightarrow{u}} & \\ X & & Y}
\end{equation*}
We will need to consider the following conditions:
\begin{enumerate}[(a)]
\item Each $B$-orbit $X_{\alpha}$ (resp. $Y_{\beta}$) contains a unique $T$-fixed point $x_\alpha$ (resp. $y_{\beta}$);
\item For each $\alpha\in I$ (resp. $\beta\in J$), the open subset $Y^{\alpha}:=\overrightarrow{u}(\overleftarrow{u}^{-1}(x_\alpha))\subset Y$ (resp. $X^{\beta}:=\overleftarrow{u}(\overrightarrow{u}^{-1}(y_\beta))\subset X$) contains a unique $T$-fixed point $y_{\hat{\alpha}}$ for some $\hat{\alpha}\in J$ (resp. $x_{\hat{\beta}}$ for some $\hat{\beta}\in I$), and is contracting to that fixed point under some one-parameter subgroup $\GG_m\subset T$ (which, of course, depends on $\alpha$ or $\beta$).
\item For each $\alpha\in I$, $\dim X_\alpha=\codim_Y Y_{\hat{\alpha}}$.
\item For each stratum $X_\alpha$, we have $H^i(\geom{X_\alpha},\Ql)=0$ for all $i>0$.
\end{enumerate}

\begin{remark}
An action of $\GG_m$ on a scheme $X$ is said to be {\em contracting to} $x\in X(k)$ if the action map extends to a map $\AA^1\times X\to X$ such that $\{0\}\times X$ is mapped to $x$.
\end{remark}

\begin{remark}
The condition (c) above implies that there is a natural bijection between the index sets $I$ and $J$: $\alpha\leftrightarrow\hat{\alpha}$ or $\hat{\beta}\leftrightarrow\beta$ characterized by the property that $(x_\alpha,y_{\hat{\alpha}})\in U$ or $(x_{\hat{\beta}},y_\beta)\in U$. 
\end{remark}

\begin{defn}
In the above setting, the {\em Radon transform from $X$ to $Y$} is the functor
\begin{equation*}
R_{X\to Y}:=\overrightarrow{u}_!\overleftarrow{u}^*\Tate{\dim Y}: \diam{X}\to \diam{Y};
\end{equation*}
with right adjoint functor
\begin{equation*}
R_{X\leftarrow Y}:=\overleftarrow{u}_*\overrightarrow{u}^!\Tate{-\dim Y}: \diam{Y}\to\diam{X}.
\end{equation*}
\end{defn}

\begin{remark}
The $B$-equivariance of the situation ensures that $R_{X\to Y}$ takes values in $\diam{Y}$. Similar remark applies to $R_{X\leftarrow Y}$.
\end{remark}

\begin{ex}
The terminology ``Radon transform'' is probably justified by the following simplest example. Let $V$ be a vector space of dimension $n$. After choosing a basis $\{v_1,\cdots,v_n\}$ for $V$, we identify $\GL(V)$ with the group $\GL_n$. Let $B$ be the subgroup of upper triangular matrices in $\GL_n$ and $T$ be the subgroup of diagonal matrices. Let $X=\PP(V)$ be the projective space parametrizing lines in $V$ and $Y=\check{\PP}(V)$ be the dual projective space parametrizing hyperplanes in $V$. Let $U=X\times Y-Z$ where $Z$ is the incidence correspondence between lines and hyperplanes. Then the $T$-fixed points in $X$ are the coordinate axes $x_i$ spanned by $v_i$ and the $T$-fixed points in $Y$ are the coordinate hyperplanes $y_i$ spanned by $\{v_j:j\neq i\}$. Condition (b) above amounts to the fact that $x_i$ is the only line which is not contained in the hyperplane $y_i$. The conditions (a)(c)(d) are also easy to verify.
\end{ex}

\begin{prop}\label{p:radon}
Under the conditions (a)(b), we have isomorphisms
\begin{equation}\label{eq:xtoy}
R_{X\to Y}(\nabla_\alpha)\cong\Delta_{\hat{\alpha}}\Tate{-\dim X_\alpha+\codim Y_{\hat{\alpha}}};
\end{equation}
\begin{equation}\label{eq:ytox}
R_{X\leftarrow Y}(\Delta_\beta)\cong\nabla_{\hat{\beta}}\Tate{\dim X_{\hat{\beta}}-\codim Y_{\beta}}.
\end{equation}
In particular, if (c) holds, then
\begin{eqnarray*}
R_{X\to Y}(\nabla_\alpha)\cong\Delta_{\hat{\alpha}};\\
R_{X\leftarrow Y}(\Delta_\beta)\cong\nabla_{\hat{\beta}}.
\end{eqnarray*}
\end{prop}
\begin{proof}
We first show (\ref{eq:xtoy}). Since all the complexes of sheaves involved are $B$-equivariant, it suffices to show that for any $\alpha\in I,\beta\in J$,
\begin{equation*}
\delta_{\beta}^*R_{X\to Y}(\nabla_{\alpha})=\left\{\begin{array}{ll} \Ql\Tate{\dim Y-\dim X_\alpha} & \alpha=\hat{\beta} \\ 0 & \alpha\neq\hat{\beta}\end{array} \right.
\end{equation*}
where $\delta_\beta$ is the inclusion $\{y_\beta\}\hookrightarrow Y$.

By proper base change, we have
\begin{eqnarray*}
\delta_\beta^*R_{X\to Y}(\nabla_\alpha) &=&\delta_\beta^*\overrightarrow{u}_!\overleftarrow{u}^*\nabla_\alpha\Tate{\dim Y}\\
&=& \HH^*_c(\geom{\overrightarrow{u}^{-1}(y_\beta)},\overleftarrow{u}^*\nabla_\alpha)\Tate{\dim Y}\\
&=& \HH^*_c(\geom{X^{\beta}},\nabla_\alpha|_{X^{\beta}})\Tate{\dim Y}.
\end{eqnarray*}
By assumption (b), under some $\GG_m\subset T$, $X^{\beta}$ is contracting to $x_{\hat{\beta}}$. Recall the following lemma (which is well-known, and a neat reference is T.A.Springer's paper \cite{Sp}, Corollary 1):
\begin{lemma}
Suppose $V$ is a scheme with a $\GG_m$-action which contracts to $a\in V(k)$, then for any complex $\calK\in D^b_m(V,\Ql)$ whose cohomology sheaves are $\GG_m$-equivariant, we have a canonical isomorphism of Frobenius modules
\begin{equation}\label{eq:res}
\HH^*(\id\to\delta_{a,*}\delta_a^{*}):\HH^*(\geom{V},\calK)\cong \delta_a^*\calK.
\end{equation}
Dually, we have
\begin{equation}\label{eq:cores}
\HH^*_c(\delta_{a,*}\delta_a^{!}\to\id):\delta_a^!\calK\cong\HH^*_c(\geom{V},\calK)
\end{equation}
Here $\delta_a:\{a\}\hookrightarrow V$ is the inclusion.
\end{lemma}

Applying (\ref{eq:cores}) to $V=X^{\beta}$ and $\calK=\nabla_{\alpha}|_{X^{\beta}}$, we get
\begin{eqnarray*}
\HH^*_c(\geom{X^{\beta}},\nabla_\alpha|_{X^{\beta}})\Tate{\dim Y} &=& \delta_{\hat{\beta}}^!(\nabla_\alpha|_{X^{\beta}})\Tate{\dim Y}\\
&=& \delta_{\hat{\beta}}^!i_{\alpha,*}\Ql\Tate{\dim X_\alpha+\dim Y}
\end{eqnarray*}
If $\alpha\neq\hat{\beta}$, then $x_{\hat{\beta}}\notin X_{\alpha}$ by assumption (a), hence the last term above is $0$. If $\alpha=\hat{\beta}$, then the last term above is the costalk of a constant sheaf on $X_\alpha$ (which is smooth) at $x_\alpha$, hence (after choosing a local orientation at $x_\alpha$) isomorphic to $\Ql\Tate{\dim X_{\alpha}+\dim Y-2\dim X_\alpha}$ $=\Ql\Tate{\dim Y-\dim X_\alpha}$.

The argument for (\ref{eq:ytox}) is dual to the above except that we have to apply (\ref{eq:res}) instead of (\ref{eq:cores}) in the final step. 
\end{proof}

\begin{cor}\label{c:equiv}
The Radon transform $R_{X\to Y}$ gives an equivalence of triangulated categories
\begin{equation*}
R_{X\to Y}:\diam{X}\to \diam{Y}
\end{equation*}
with $R_{X\leftarrow Y}$ as quasi-inverse.
\end{cor}
\begin{proof}
The adjunction transform $\id\to R_{X\leftarrow Y}\circ R_{X\to Y}$ (resp. $R_{X\to Y}\circ R_{X\leftarrow Y}\to\id$) gives isomorphisms on the generating objects: integer-weight-twists of $\nabla_{\alpha}$ (resp. $\Delta_{\beta}$), by Proposition \ref{p:radon}.
\end{proof}

\subsection{Mixed tilting sheaves under the Radon transform}
\begin{prop}[Geometric Ringel duality]\label{p:tiltproj}
Suppose the conditions (a)--(d) in Section \ref{ss:radon} hold. 
\begin{enumerate}[(1)]
\item For any mixed tilting sheaf $\calT\in\pdiam{X}$, $\omega(R_{X\to Y}(\calT))$ is a projective object in $\pdiac{Y}$;
\item For any indecomposable mixed tilting extension $\calT$ of $\Qla$, \newline $\omega(R_{X\to Y}(\calT))$ is a projective cover of $\omega(\IC_{\hat{\alpha}})$ in $\pdiac{Y}$. Moreover, $\IC_{\hat{\alpha}}$ is the unique quotient of $R_{X\to Y}(\calT)$ in $\Perm{Y}$ whose underlying non-mixed perverse sheaf is semisimple.
\end{enumerate}
\end{prop}
\begin{proof}
The argument for (1) is essentially borrowed from \cite{BBM}, Section 2.3. Let $\calP=R_{X\to Y}(\calT)$.
Since $\calT$ has a $\nabla$-flag, $\calP$ has a $\Delta$-flag by Proposition \ref{p:radon}, hence $\calP\in\pdiam{Y}$.

Next, we show that $\omega(\calP)$ is a projective object in $\pdiac{Y}$. Since every object in $\pdiac{Y}$ is a successive extension of $\omega(\Delta_\beta)[d]$ for $d\geq0$, it suffices to show that $\Ext_{\geom{Y}}^i(\calP,\Delta_\beta)=0$ for all $i>0$. By adjunction and Proposition \ref{p:radon},
\begin{eqnarray*}
\Ext_{\geom{Y}}^i(\calP,\Delta_\beta)&=&\Ext_{\geom{X}}^i(\calT,R_{X\leftarrow Y}(\Delta_\beta))\\
&=&\Ext_{\geom{X}}^i(\calT,\nabla_{\hat{\beta}}).
\end{eqnarray*}
The last term above is 0 because $\calT$ has a $\Delta$-flag and
\begin{equation*}
\Ext_{\geom{X}}^i(\Delta_\gamma,\nabla_{\hat{\beta}})=\left\{\begin{array}{ll}0,&\textup{if }\gamma\neq\hat{\beta}\\ H^i(\geom{X_\gamma},\Ql)=0,&\textup{if }\gamma=\hat{\beta}\end{array}\right.
\end{equation*}
by condition (d).

(2) By Corollary \ref{c:equiv},
\begin{equation*}
\End_{\geom{Y}}(\calP)=\End_{\geom{X}}(\calT).
\end{equation*}
has no nontrivial idempotents because $\omega(\calT)$ is indecomposable, hence $\omega(\calP)$ is also indecomposable. Therefore it is a projective cover of an IC sheaf in $\pdiac{Y}$. Note that we have a surjection $\calT\twoheadrightarrow\nabla_\alpha$ in $\pdiam{X}$ whose kernel has a $\nabla$-flag. Therefore, by Proposition \ref{p:radon}, we have a surjection $\calP\twoheadrightarrow\Delta_{\hat{\alpha}}$ in $\pdiam{Y}$ whose kernel has a $\Delta$-flag. In particular, we get a surjection $\calP\twoheadrightarrow\Delta_{\hat{\alpha}}\twoheadrightarrow\IC_{\hat{\alpha}}$ in $\pdiam{Y}$. This implies that $\omega(\calP)$ is a projective cover of $\omega(\IC_{\hat{\alpha}})$. 

Suppose $\calP\twoheadrightarrow\calQ\in\Perm{Y}$ and $\omega(\calQ)\in\pdiac{Y}$ is semisimple, then the property of projective covers implies that $\omega(\calQ)\cong\omega(\IC_{\hat{\alpha}})$. Let $\calI$ be the image of $\calP\to\calQ\oplus\IC_{\hat{\alpha}}$. Then we also have $\omega(\calI)\cong\omega(\IC_{\hat{\alpha}})$, and the two projections give $\calI\cong\calQ$ and $\calI\cong\IC_{\hat{\alpha}}$. Hence $\calQ\cong\IC_{\hat{\alpha}}$. The proof is complete.
\end{proof}

\begin{theorem}\label{th:tcan}
Suppose the conditions (a)--(d) in Section \ref{ss:radon} hold. Let $\calT$ be an indecomposable mixed tilting extension of $\Qla$. Then $\calT$ satisfies the non-cancellation property.
\end{theorem}
\begin{proof}
Let $\calP=R_{X\to Y}(\calT)$. According to Proposition \ref{p:strong}, it suffices to show that the Frobenius weights on $\End^0_{\geom{X}}(\calT)\cong\End^0_{\geom{Y}}(\calP)=\ker(\End_{\geom{Y}}(\calP)\to\End_{\geom{Y_{\hat{\alpha}}}}(i_{\hat\alpha}^*\calP))$ are negative.

Consider the weight filtration $w_{\leq i}\calP$ of $\calP$. We first claim that each $\omega(\Gr^w_i\calP)$ is in $\pdiac{Y}$. In fact, since $\calP$ has a $\Delta$-flag, it suffices to show that $\omega(\Gr^w_i\Delta_\beta)\in\pdiac{Y}$ for each $\beta$. Since $\Delta_\beta$ is $B$-equivariant, so are the $\Gr^w_i\Delta_\beta$, but it is easy to see that any $B$-equivariant perverse sheaf is in $\pdiac{Y}$.

By \cite{BBD}, Th\'{e}or\`{e}me 5.3.8, each perverse sheaf $\omega(\Gr^w_i\calP)$ is semisimple. By Proposition \ref{p:tiltproj}(c), we see the last piece of $\Gr^w_i\calP$ is $\IC_{\hat{\alpha}}$. We conclude that $\calP$ has weight $\leq 0$ and $\Gr^w_0\calP=\IC_{\hat{\alpha}}$. Since $\omega(\calP)$ is projective, the functor $\Hom_{\geom{Y}}(\calP,-)$ is exact. Therefore $\End^0_{\geom{Y}}(\calP)$ is a successive extension of the Frobenius modules $V_i:=\Hom_{\geom{Y}}(\calP,\Gr^w_i\calP)$ for $i<0$. Note that since $\omega(\Gr^w_i\calP)$ is semisimple, we have a decomposition in $\pdiam{Y}$
\begin{equation}
\Gr^w_i\calP=V_i\otimes\IC_{\hat{\alpha}}\bigoplus\calQ_i
\end{equation}
where $\calQ_i$ does not have simple constituents isomorphic to twists of $\IC_{\hat{\alpha}}$. Since $\Gr^w_i\calP$ has weight $i$, hence $V_i$ also has weight $i<0$. Therefore $\End^0_{\geom{Y}}(\calP)$ has negative weights. The proof is complete.
\end{proof}

\begin{remark}
Under the above assumptions, by the remarks following Proposition \ref{p:can}, the indecomposable mixed tilting extension of $\Qla$ is unique up to isomorphisms (which are unique up to a scalar). We denote it by $\calT_\alpha$.
\end{remark}

\subsection{Calculation of weights III--inverse matrix}\label{ss:cal2}
Suppose the conditions (a)--(d) in Section \ref{ss:radon} hold. We now give another method for computing the weight polynomials of $\calT_\alpha$, which is a composition of Ringel duality (see Proposition \ref{p:tiltproj}) and Bernstein-Gelfand-Gelfand reciprocity.

\begin{prop}\label{p:tp}
The matrix $(W_{\gamma}(\calT_\alpha,t))_{\alpha,\gamma\in I}$ is the inverse of the matrix\newline $(W_{\hat{\gamma}}(\IC_{\hat{\alpha}},t^{-1}))_{\alpha,\gamma\in I}$.
\end{prop}
\begin{proof}
Consider the isomorphism
\begin{equation*}
R_\#:K(\diam{X})\stackrel{\sim}{\to} K(\diam{Y})
\end{equation*}
induced by the Radon transform $R_{X\to Y}$. Let $\calP_{\hat{\alpha}}=R_{X\to Y}(\calT_\alpha)$. Since $R_\#[\nabla_{\gamma}]=[\Delta_{\hat{\gamma}}]$ by Proposition \ref{p:radon}, we have
\begin{eqnarray*}
[\calP_{\hat{\alpha}}]=R_\#[\calT_\alpha]&=&R_\#\left(\sum_{\gamma\leq\alpha}W_\gamma(\calT_\alpha,t^{-1})[\nabla_\gamma]\right)\\
&=&\sum_{\gamma\leq\alpha}W_\gamma(\calT_\alpha,t^{-1})[\Delta_{\hat{\gamma}}]
\end{eqnarray*}
Therefore
\begin{equation}\label{eq:tp}
W_{\hat{\gamma}}(\calP_{\hat{\alpha}},t^{-1})=W_\gamma(\calT_\alpha,t)
\end{equation}
On the other hand, $W_{\hat{\gamma}}(\calP_{\hat{\alpha}},t^{-1})$ is the weight polynomial of $\Hom_{\geom{Y}}(\calP_{\hat{\alpha}},\nabla_{\hat{\gamma}})$, viewed as a mixed complex on $\Spec(k)$. Since the functor $\Hom_{\geom{Y}}(\calP_{\hat{\alpha}},-)$ extracts the simple constituents isomorphic to a twist of $\IC_{\hat{\alpha}}$, the weight polynomial of $\Hom_{\geom{Y}}(\calP_{\hat{\alpha}},\nabla_{\hat{\gamma}})$ is the same as the weighted multiplicity of $\IC_{\hat{\alpha}}$ in the composition series of $\nabla_{\hat{\gamma}}$ (This is the BGG reciprocity). Therefore the matrix $W_{\hat{\gamma}}(\calP_{\hat{\alpha}},t^{-1})$ is the same as the matrix expressing $[\nabla_{\hat{\gamma}}]$ in terms of $[\IC_{\hat{\alpha}}]$, hence inverse to the matrix expressing $[\IC_{\hat{\alpha}}]$ in terms of $[\nabla_{\hat{\gamma}}]$. Since $\IC_{\hat{\alpha}}$ is Verdier self-dual, we conclude that the matrix $(W_{\hat{\gamma}}(\calP_{\hat{\alpha}},t^{-1}))$ is the inverse matrix of $(W_{\hat{\gamma}}(\IC_{\hat{\alpha}},t^{-1}))$, which, together with (\ref{eq:tp}), implies the proposition.
\end{proof}


\section{Flag and affine flag varieties}\label{s:ex}

Let $G$ be a split reductive group over $k$. Fix a pair of opposite Borel subgroups $B^+$ and $B^-$ whose intersection is a split maximal torus $T$. Let $\coch(T)$ be the cocharacter group of $T$. Let $W$ be the Weyl group determined by $T$, then $W$ has a set of simple reflections determined by $B^+$ and hence a length function $\ell:W\to\ZZ_{\geq0}$. Let $w_0\in W$ be the longest element. Let $2\check{\rho}\coch(T)$ be the sum of positive coroots, viewed as a one-parameter subgroup of $T$. Let $\theta$ be the highest root.

\subsection{Radon transform for the flag variety}
The Radon transform for the flag variety was considered in \cite{BBM}, Section 2.2. We briefly recall it here. 

We consider two isomorphic flag varieties $X=G/B^+$ and $Y=G/B^-$. It is well-known that the orbit $U$ of $(B^+/B^+,B^-/B^-)$ under the diagonal $G$ action on the product $X\times Y$ is open dense. In fact, $U$ consists of pairs of opposite Borel subgroups of $G$. Consider $U$ as a correspondence between $X$ and $Y$, with two projections $\overleftarrow{u}:U\to X$ and $\overrightarrow{u}:U\to Y$. For $w\in W$, let $X_w=B^+wB^+/B^+$ and $Y_w=B^+wB^-/B^-$. We have the stratification by $B^+$-orbits
\begin{equation*}
X=\bigsqcup_{w\in W}X_w;\hspace{1cm}Y=\bigsqcup_{w\in W}Y_w.
\end{equation*}

We verify the conditions in Section \ref{ss:radon}.

(a) holds because $X_w$ (resp. $Y_w$) contracts to $wB^+/B^+$ (resp. $wB^-/B^-$) under the one-parameter subgroup $2\check{\rho}\in\coch(T)$.

(b) For each $w\in W$, $\overrightarrow{u}^{-1}(wB^-)=wB^-B^+/B^+$ contracts to $wB^+/B^+$ under the one-parameter subgroup $Ad(w)(-2\check{\rho})$; $\overleftarrow{u}^{-1}(wB^+)=wB^+B^-/B^-$ contracts to $wB^-/B^-$ under the one-parameter subgroup $Ad(w)(2\check{\rho})$.

(c)(d) follow from the fact that $X_w\cong\AA^{\ell(w)}$ and $Y_w\cong\AA^{\ell(w_0)-\ell(w)}$.

Therefore the results of Section \ref{s:geom} apply to this situation.

\subsection{Radon transform for the affine flag variety}\label{ss:aff}
In this section, we will have to deal with ind-schemes and schemes of infinite type.

Let $F$ be the field of formal Laurent polynomials $k((z))$ and $\calO^+=k[[z]]$ be the valuation ring. Let $\calO^-=k[z^{-1}]\subset F$. Let $I^+\subset G(\calO^+)$ be the Iwahori subgroup given by the preimage of $B^+$ under the projection $G(\calO^+)\to G$. Let $I^-\subset G(\calO^-)$ be the preimage of $B^-$ under the projection $G(\calO^-)\to G$. We call a subgroup of $G(F)$ which is conjugate to $I^-$ a {\em co-Iwahori subgroup}. 

Let $X=G(F)/I^+$ be the {\em affine flag variety}. This is an ind-scheme locally of finite type parametrizing $G$-torsors over $\PP^1$ together with a trivialization on $\PP^1-\{0\}$ and a $B^+$-reduction at $\{0\}$. Let $Y=G(F)/I^-$ be the {\em thick affine flag variety}. This is a scheme of infinite type parametrizing $G$-torsors over $\PP^1$ together with a full level structure at $\{0\}$ and a $B^-$-reduction at $\{\infty\}$. For the basic properties of the thick affine flag variety, we refer to \cite{Ka}.

Similar to the finite situation, the $G(F)$-orbit $U$ of $(I^+,I^-)$ in $X\times Y$ is open dense, parametrizing pairs of ``opposite'' Iwahori and co-Iwahori subgroups in $G(F)$. 

Consider the action of the one-dimensional torus $\GG_m^{\rot}$ on $X$ and $Y$ by loop rotations: $s\cdot g(z)=g(sz)$ for $s\in\GG_m^{\rot}$ and $g(z)\in G(F)$. We denote $\tilT=T\times\GG_m^{\rot}$. This is the torus we are going to consider. Let $\tilIp=I^+\rtimes\GG_m^{\rot}$. It acts on $X,Y$ and $U$.

The $I^+$-orbits (which are the same as the $\tilIp$-orbits) on either $X$ or $Y$ are parametrized by the affine Weyl group $\tilW=\coch(T)\rtimes W$. For $\tilw\in\tilW$, let $X_{\tilw}=I^+\tilw I^+/I^+$ and $Y_{\tilw}=I^+\tilw I^-/I^-$. We have
\begin{equation*}
X=\bigsqcup_{\tilw\in \tilW}X_{\tilw};\hspace{1cm}Y=\bigsqcup_{\tilw\in \tilW}Y_{\tilw}.
\end{equation*}
The affine Weyl group has a partial order such that $\tilw\leq\tilw'$ $\Leftrightarrow$ $X_{\tilw}\subset\overline{X_{\tilw'}}$ $\Leftrightarrow$ $\overline{Y_{\tilw}}\supset Y_{\tilw'}$.

In order to fit into the framework of Section \ref{ss:radon}, we have to do certain truncations to these spaces. Fix $\tilu\in \tilW$. Consider
\begin{equation*}
X_{\leq\tilu}=\bigsqcup_{\tilw\leq\tilu}X_{\tilw};\hspace{1 cm} Y_{\leq\tilu}=\bigsqcup_{\tilw\leq\tilu}Y_{\tilw}.
\end{equation*}
Then $X_{\leq\tilu}$ is a closed (in fact projective) subscheme of $X$, while $Y_{\leq\tilu}$ is an open subscheme of $Y$. Recall that there is a principal congruence subgroup $K\subset G(\calO^+)$ (depending on $\tilu$) which acts freely on $Y_{\leq\tilu}$ and acts trivially on $X_{\leq\tilu}$ (cf. \cite{Ka}). Let $Z$ be the quotient $K\backslash Y_{\leq\tilw}$. We remark that $Z$ is a scheme of finite type which parametrizes $G$-torsor over $\PP^1$ with a $K$-level structure at $\{0\}$ and a $B^-$-reduction at $\{\infty\}$. Since $K$ is normal in $\tilIp$, the group $\tilIp/K$ acts on $Z$ and $X_{\leq\tilu}$, and $Z$ is stratified by finitely many $\tilIp$-orbits $Z_{\tilw}=K\backslash Y_{\tilw}$, for $\tilw\leq\tilu$. Let $U_{\leq\tilu}=U\cap(X_{\leq\tilu}\times Y_{\leq\tilu})$. The diagonal action of $K$ on $U_{\leq\tilu}$ is still free so that we can form the quotient scheme $V_{\leq\tilu}=K\backslash U_{\leq\tilu}$. We now get an $\tilIp/K$-equivariant correspondence
\begin{equation*}
\xymatrix{& V_{\leq\tilu}\ar[dl]_{\overleftarrow{v}}\ar[dr]^{\overrightarrow{v}} & \\ X_{\leq\tilu} & & Z}.
\end{equation*}

We verify the conditions in Section \ref{ss:radon}. 

(a) holds because $X_{\tilw}$ (resp. $Z_{\tilw}$) contracts to $\tilw I^+/I^+$ (resp. $\tilw I^-/I^-$) under the one-parameter subgroup $(2\check{\rho},1+\langle2\check{\rho},\theta\rangle)\in\coch(\tilT)=\coch(T)\oplus\ZZ$.

(b) We first note that there is a natural action $Ad$ of $\tilW$ on $\tilT$:
\begin{equation*}
Ad(\tilw)(t,s)=(Ad(w)t\cdot s^{-Ad(w)\lambda},s),\hspace{.5cm}\tilw=(\lambda,w),(t,s)\in\tilT.
\end{equation*}
It is easy to verify that for $\tilw\in\tilW$ and $\tilt\in\tilT$, we have
\begin{equation*}
Ad(\tilw)(\tilt)\circ\tilw=\tilw\circ\tilt.
\end{equation*}
as left translation actions on $X$ or $Y$. The action $Ad$ also induces an action of $\tilw$ on $\coch(\tilT)\cong\coch(T)\oplus\ZZ$. 
 
Now we verify (b). For each $\tilw\in\tilW$, $\overrightarrow{v}^{-1}(\tilw I^-)=\tilw I^-I^+/I^+$ contracts to $\tilw I^+/I^+$ under the one-parameter subgroup $Ad(\tilw)(-2\check{\rho},-1-\langle2\check{\rho},\theta\rangle)$; $\overleftarrow{v}^{-1}(\tilw I^+)=\tilw I^+I^-/I^-$ contracts to $\tilw I^-/I^-$ under the one-parameter subgroup $Ad(\tilw)(2\check{\rho},1+\langle2\check{\rho},\theta\rangle)$.

(c) follows from the fact that $\dim X_{\tilw}=\ell(\tilw)$ and $\codim_Y Y_{\tilw}=\codim_Z Z_{\tilw}=\ell(\tilw)$.

(d) Both $X_{\tilw}$ and $Z_{\tilw}$ are finite dimensional homogeneous spaces under the unipotent radical of $I^+/K$. They both contain a $k$-point (the unique $T$-fixed point), hence they are isomorphic to affine spaces.

Therefore the results of Section \ref{s:geom} apply to this situation as well. Note that we can choose $\tilu$ large enough for our purposes. 

\subsection{Identification of weight polynomials}
Let $X$ be the affine flag variety as above. According to the remarks following Theorem \ref{th:tcan}, we can speak about {\em the} indecomposable mixed tilting extension $\calT_{\tilw}$ of the constant perverse sheaf $\Ql\Tate{\ell(\tilw)}$ on $X_{\tilw}$ for $\tilw\in\tilW$.  
\begin{theorem}\label{th:KL}
The weight polynomials of $\calT_{\tilw}$ are
\begin{equation}\label{eq:tiltKL}
W_{\tilv}(\calT_{\tilw},t)=t^{\ell(\tilw)-\ell(\tilv)}\cdot P_{\tilv,\tilw}(t^{-2})
\end{equation}
where $P_{\tilv,\tilw}$ are the Kazhdan-Lusztig polynomials for $\tilW$ (cf. \cite{KL}, Theorem 1.1). Moreover, $\calT_{\tilw}$ satisfies the condition \textup{(W)}.
\end{theorem}
\begin{proof}
Note that $K(\diam{X})$ naturally maps to the affine Hecke algebra $\calH_{\tilW}$ of $\tilW$. Recall that Theorem 1.1 of \cite{KL} says that the Kazhdan-Lusztig basis elements $C_{\tilw}$ for $\calH_{\tilW}$ are self-dual and satisfy
\begin{equation*}
(-1)^{\ell(\tilw)}C_{\tilw}=\sum_{\tilv\leq\tilw}t^{\ell(\tilw)-\ell(\tilv)}\cdot P_{\tilv,\tilw}(t^{-2})[\Delta_{\tilw}].
\end{equation*}
Notice that the standard basis used in \textit{loc.cit.} is equal to $(-t)^{\ell(\tilw)}[\Delta_{\tilw}]$.

The Kazhdan-Lusztig conjecture (which is a theorem of Beilinson-Bernstein and Kashiwara-Tanisaki) says that $P_{\tilv,\tilw}$ is a polynomial with non-negative integer coefficients of degree $\leq\dfrac{1}{2}(\ell(\tilw)-\ell(\tilv)-1)$ for $\tilv<\tilw$. Therefore $\{t^{\ell(\tilw)-\ell(\tilv)}\cdot P_{\tilv,\tilw}(t^{-2})\in\ZZ_{\geq0}[t]\}$ is a solution to the self-duality equation (\ref{eq:selfdual}) satisfying the condition (W'), hence the non-cancellation property. The initial value condition is also satisfied since $P_{\tilw,\tilw}=1$. On the other hand, by Theorem \ref{th:tcan}, $\calT_{\tilw}$ has the non-cancellation property. Hence $\{W_{\tilv}(\calT_{\tilw},t)\}$ is also a non-cancellation solution to the self-duality equation with the correct initial value. Therefore, the theorem follows by the uniqueness statement proved in Proposition \ref{p:uni}.
\end{proof}

\begin{cor}
Similar identities hold if $X$ is replaced by the flag variety.
\end{cor}
\begin{proof}
This can either be proved independently by using the Radon transform for the flag varieties (argue as above), or by restricting the equation (\ref{eq:tiltKL}) to elements $v,w\in W$. 
\end{proof}

\subsection{Partial flag and affine partial flag varieties}

The Radon transforms also exist for partial flag and affine partial flag varieties. Although we do not actually need it to compute the weights of mixed tilting sheaves, we nevertheless sketch the construction in the affine case. Let $\tau$ be the Chevalley involution of $G$ which sends the root space corresponding to a root $\alpha$ to the root space corresponding to $-\alpha$. Let $\sigma$ be the involution of $G(k[z,z^{-1}])$ defined by $g(z)\mapsto \tau(g(z^{-1}))$. Then $\sigma$ sends the root space corresponding to an affine root $\tilalpha$ to the root space corresponding to $-\tilalpha$. Let $P^+$ be a parahoric subgroup of $G(F)$ containing $I^+$. Let $P^-:=\sigma(P^+\cap G(k[z,z^{-1}]))$. Let $X=G(F)/P^+$ be the affine partial flag variety (ind-scheme locally of finite type) and $Y=G(F)/P^-$ be the thick affine partial flag variety (scheme of infinite type). The $G(F)$-orbit $U$ of the point $(P^+/P^+,P^-/P^-)\in X\times Y$ is dense open. We view $U$ as an $I^+$-equivariant correspondence between $X$ and $Y$. Then the truncation construction in Section \ref{ss:aff} has an obvious analogue here and we can similarly check the conditions (a)-(d). As a consequence, the results of Section \ref{s:geom} apply to mixed tilting sheaves on $X$. In particular, we can speak about {\em the} indecomposable mixed tilting extension of the constant perverse sheaf on some stratum $X_{\tilw}$, where $\tilw\in\tilW/\tilW_{P^+}$. 

\begin{cor}[of Theorem \ref{th:KL}]
The mixed tilting sheaf $\calT_{\tilw}$ on the affine partial flag variety $G(F)/P^+$ satisfies the weight condition \textup{(W)}.
\end{cor}
\begin{proof}
Choose a lift $\tilu$ of $\tilw$ in $\tilW$ such that $\ell(\tilu)$ is minimal in the coset $\tilu\tilW_{P^+}$. Consider the projection
\begin{equation*}
f:(G(F)/I^+)_{\leq\tilu}\to(G(F)/P^+)_{\leq\tilw}.
\end{equation*}
It is easy to verify the conditions in Section \ref{ss:wpush}, hence Proposition \ref{p:wpush} applies. In particular, $f_!\calT_{\tilu}=\calT_{\tilw}$ satisfies the condition (W).
\end{proof}

Similar statements for partial flag varieties $G/P$ also hold. To explicitly calculate the weight polynomials in these situations, we can either use push-forward from (affine) flag varieties (Proposition \ref{p:wpush}) or the inverse matrix method (Proposition \ref{p:tp}).

\subsection*{Acknowledgment}
The author is grateful to R.Bezrukavnikov for the important observation that the push-forward of a tilting sheaf is usually a tilting sheaf. He also thanks M.Goresky, R.MacPherson, D.Nadler, J.Suh, D.Treumann and X.Zhu for encouragement and helpful suggestions.

\end{document}